\documentclass[12pt]{article}

\usepackage{amsmath}
\usepackage{amssymb}
\usepackage{amsthm}
\usepackage{marvosym}
\usepackage{graphicx}

\newtheorem{theorem}{Theorem}

\newtheorem{lemma}[theorem]{Lemma}
\newtheorem{corollary}[theorem]{Corollary}

\newtheorem{question}{Question}

\DeclareMathOperator\iw{iw}
\DeclareMathOperator\id{id}
\DeclareMathOperator\Gal{Gal}
\DeclareMathOperator\width{width}

\title{Computing the Size of Intervals in the Weak Bruhat Order}
\author{Joshua Cooper and Anna Kirkpatrick \\ Mathematics Department, University of South Carolina \\ 1523 Greene St., Columbia SC 29208}

\begin{document}

\maketitle


\begin{abstract}
The {\em weak Bruhat order} on \( { \mathcal S }_n\) is the partial order \(\prec\) so that \(\sigma \prec \tau\) whenever the set of inversions of \(\sigma\) is a subset of the set of inversions of \(\tau\).  We investigate the time complexity of computing the size of
intervals with respect to \(\prec\). Using relationships between two-dimensional posets and the
weak Bruhat order, we show that the size of the interval \( [ \sigma_1
, \sigma_2 ]\) can be computed in polynomial time whenever \(
\sigma_1^{-1} \sigma_2\) has bounded width (length of its longest decreasing subsequence) or bounded intrinsic width (maximum width of any non-monotone permutation in its block decomposition).  Since permutations of intrinsic width \(1\) are precisely the {\em separable} permutations, this greatly extends a result of Wei.
Additionally, we show that, for large \(
n\), all but a vanishing fraction of permutations
\( \sigma\) in \( { \mathcal S }_n\) give rise to intervals \( [ \id ,
\sigma ]\) whose sizes can be computed with a sub-exponential time
algorithm. The general question of the difficulty of computing the size
of arbitrary intervals remains open.
\end{abstract}

\section{Introduction and Definitions }
A {\em permutation} of \( n\) is a bijective function from \( [ n ] = \{
1 , 2 , \ldots , n \}\) onto itself; we write \( {\mathcal S }_n\) for
the set of all permutations of \( n\). We denote the composition of two permutations,
\( f\) and \( g\), by \( f g\), i.e., \( fg ( x ) = f ( g ( x
) )\). A permutation \( \sigma\) is sometimes identified with the string
\( \sigma ( 1 ) \sigma ( 2 ) \ldots \sigma ( n )\), its expression in
so-called {\em one-line notation}.
The {\em width} of a permutation \( \sigma\) is the length of its longest
decreasing subsequence, denoted \( \width ( \sigma )\).
 A {\em transposition} reversing \( k\) and \( l\) in \([n]\) is a permutation \(
\sigma\) with \( \sigma ( k ) = l\), \( \sigma ( l ) = k\), and \(
\sigma ( i ) = i\) for \( i \not \in \{ l , k \}\). If \( |k-l| = 1\), then we say that \(
\sigma\) is an {\em adjacent transposition}.

A {\em poset of size \( n\)} is a pair \( P = ( S , \prec )\) where \( S\) is
a set of cardinality \( n\), and \( \prec\) is a partial order
relation, that is, a binary relation which is reflexive, transitive,
and antisymmetric. If \( p \prec
q\) or \( q \prec p\), then we say that \( p\) and \( q\) are
{\em comparable,} and {\em incomparable} otherwise.  
Given a poset \( P = ( S , \prec )\), a {\em linear extension} of \( P\) is a
total order \( \prec^\prime\) on \( S\) such that, for all \( p , q \in
S\), \( p \prec q\) implies \( p \prec^\prime q\). Equivalently, \(
\prec^\prime\) is a total order with \( \prec \subseteq
\prec^\prime\) as relations. The set of all linear extensions of \( P\) is denoted
\( { \mathcal L } ( P )\). An {\em interval} \( [ p , q ]\) of a poset \(
P = ( S , \prec )\) is the set \( \{ r \in S | p \prec r \prec q \}\).
A {\em chain} \( C \subseteq S\) is a subset of pairwise comparable elements, and an
{\em antichain} \( A \subseteq S\) is a set of pairwise incomparable
elements. The {\em width} of a poset \( P\) is the cardinality of its
largest antichain.

Given a poset \( P = ( S , \prec )\), a collection of linear extensions \(
{ \mathcal R } = \{ \prec_1 , \prec_2 , \ldots , \prec_k \}\) is called
a {\em realizer} of \( P\) if \( \prec = \bigcap_{ i = 1 }^k \prec_k\),
where each relation \( \prec_i\) is interpreted as a set of ordered
pairs. Equivalently, \( { \mathcal
R }\) is a realizer of \( P\) if, for all \( p , q \in S\), \( p \prec
q\) if and only if \( p \prec_i q\) for all \( 1 \leq i \leq k\). A
realizer is said to be {\em minimal} if it has the smallest possible
cardinality among all realizers of \( P\). The {\em dimension} of \( P\)
is the cardinality of any minimal realizer.

We now define the {\em weak Bruhat order} on permutations of \( n\), an
important structure in algebraic combinatorics and elsewhere.
Given two permutations \( \sigma_1\) and \( \sigma_2\), we have the
covering relation \( \sigma_1 \lessdot \sigma_2\) if and only if there
is an adjacent transposition \( \tau\) reversing \( k\) and \( l\) with
\( \sigma_1 = \sigma_2 \tau\), \( k < l\), and \( \sigma_1 ( k ) <
\sigma_1 ( l )\). The reflexive and transitive closure of \( \lessdot\)
gives the partial ordering relation for the weak Bruhat order. The weak
Bruhat order on \( { \mathcal S }_3\) is depicted in Figure
\ref{Bruhat_order_example}.

\begin{figure}[htb]
\centering
\includegraphics[scale = 0.5] {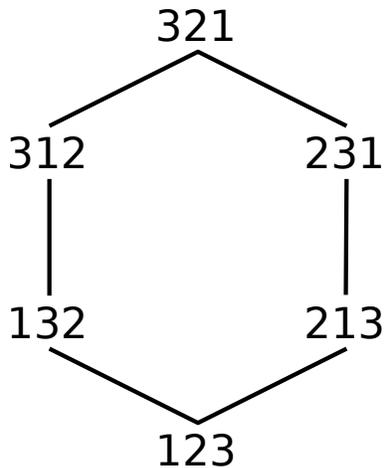}
\caption{The weak Bruhat order on \( { \mathcal S }_3\)}
\label{Bruhat_order_example}
\end{figure}

The present work investigates computational complexity questions
regarding the structures we have just defined.
We refer to standard texts in computational complexity theory to precisely define hardness of decision and functional complexity problems (e.g., \cite{Arora}).  Roughly, a decision problem is in \({\bf P}\) if the answer can be computed in polynomial time (in the size of the input instance); it is in \({\bf NP }\) if the answer can be certified in polynomial time; it is \({\bf NP }\)-hard if every problem in \({\bf NP }\) can be reduced to it in polynomial time (i.e., it is at least as hard as all problems in \({\bf NP }\)).  Similarly, a function problem (a computational problem whose output is an integer instead of only a single bit) is in \({\bf FP}\) if the answer can be computed in polynomial time (in the size of the input instance); it is in \(\#{\bf P}\) if it consists of computing the number of correct solutions to a problem in \({\bf NP }\); it is \(\#{\bf P}\)-hard if the problem of computing the number of correct solutions to any problem in \({\bf NP }\) can be reduced to this problem in polynomial time.


These two ideas, the weak Bruhat order on \( { \mathcal S }_n\) and
poset linear extensions, give rise to many computational complexity
questions. We address two in particular. First, how hard is it to
compute interval sizes in the weak Bruhat order? In particular,
 given permutations \( \sigma_1\) and \( \sigma_2\) of \( n\), what is
the computational complexity of computing the size of the interval \( [
\sigma_1 , \sigma_2 ]\) in the weak Bruhat order on \( { \mathcal S
}_n\)? Wei gives an explicit formula for the case when \( \sigma_2
\sigma_1^{ - 1 }\) is a separable permutation \cite{Wei}, proving that
the size of such intervals can be computed in polynomial time. Second,
given a poset, can we compute the number of linear extensions? In
general, this task is known to be \( \# {\bf P}\)-hard
\cite{Brightwell}. However, several restricted cases of this task are
known to be possible in polynomial time. For example, if \( P\) has
bounded width or bounded intrinsic width, then the number of linear
extensions of \( P\) can be computed in polynomial time \cite{Cooper}.
We apply results on linear extension enumeration to the computation of the sizes of intervals in the weak Bruhat order.
In particular, we prove the polynomial time
computability of a much larger set of intervals than the set addressed by
Wei \cite{Wei}. Most generally, we are able to calculate in polynomial
time the size of \( [ \sigma_1 , \sigma_2 ]\) whenever \( \sigma_1^{ -
1 } \sigma_2\) has bounded intrinsic width.

One crucial tool we use is a bijection between
permutations of \( n\) and two-dimensional posets with ground set \( [
n ]\), where the labeling provided by the ground set is a linear
extension (which Bj\"{o}rner and Wachs in \cite{Bjorner} call a
{\em natural labeling}). Let \( { \mathcal P }_n\) be the set of all
such posets. We define \( \Phi : { \mathcal S }_n \rightarrow { \mathcal P }_n\) as
follows. Let \( \sigma\) be a permutation of \( n\), and consider \(
\sigma\) written in one-line notation. Interpret \( \sigma\) as one
chain in a realizer, reading left-to-right. For the other chain, take \( [ n ]\) with the
usual \( \leq\) ordering. This realizer yields a two-dimensional poset,
\( \Phi ( \sigma )\). As an example, consider the permutation \( \sigma
= 24135\). To construct \( \Phi ( \sigma )\), we consider the realizer
composed of the two chains \( 1 \prec 2 \prec 3 \prec 4 \prec 5\) and
\( 2 \prec 4 \prec 1 \prec 3 \prec 5\). Their intersection gives the
poset whose Hasse diagram is pictured in Figure \ref{basic_example} (a). Note that this
poset can also be obtained from the graph of \( \sigma\) as a function,
as depicted in Figure 
\ref{basic_example} (b).

\begin{figure}[htb]
\centering
\includegraphics[scale = 0.5] {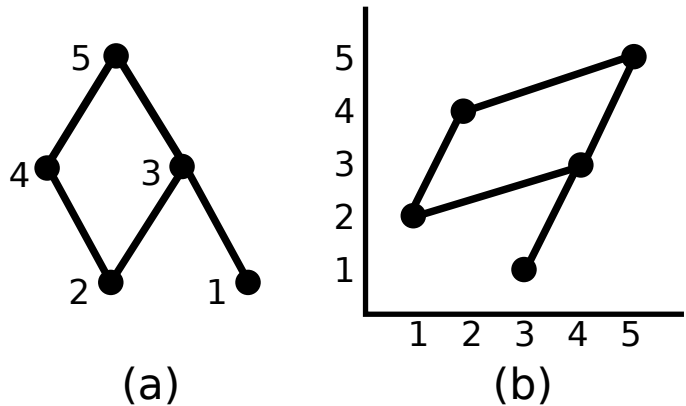}
\caption{A two-dimensional poset and its corresponding permutation. }
\label{basic_example}
\end{figure}

Some other key relationships between posets and permutations are discussed 
at the beginning of Section \ref{bounded_width_section}.
This material builds toward the main result of Section
\ref{bounded_width_section}, that \( | [ \sigma_1 , \sigma_2 ] |\) is
polynomial-time computable whenever \( \sigma_1^{-1} \sigma_2\) has
bounded width. Section \ref{intrinsic_width_section} defines
intrinsic width in order to generalize the argument of Section
\ref{bounded_width_section} to include the more general case where \( \sigma_1^{-1}
\sigma_2\) has bounded intrinsic width. Finally, Section
\ref{random_permutations_section} applies the ideas of previous
sections to random permutations.

A central question that remains is the following.

\begin{question} What is the computational complexity of computing the number of linear extensions of a dimension-two poset, or, equivalently, of the size of intervals in the weak Bruhat order?
\end{question}

\section{Permutations of Bounded Width}
\label{bounded_width_section}

We make significant use of a result of Bj{\"o}rner and Wachs
relating the linear extensions of a two-dimensional poset to an
interval in the weak Bruhat order.

\begin{theorem}[Bj{\"o}rner, Wachs \cite{Bjorner}]
\label{Bruhat_intervals_are_poset_extensions}
Let \( U \subseteq { \mathcal S }_n\). Then \( U\) is an interval in
the weak Bruhat order if and only if there is a poset \( P\) with
ground set \( [ n ]\) such that \( U\) consists exactly of the linear
extensions of \( P\).
\end{theorem}

The following lemma is implicit in the work of Bj{\"o}rner and Wachs.
It is stated separately here to lend clarity to subsequent proofs.

\begin{lemma}
\label{Bruhat_intervals_from_identity}
Let \( \sigma \in { \mathcal S }_n\). Consider \( [ \id , \sigma ]\) as
an interval in the weak Bruhat order. Then, \( { \mathcal L } ( \Phi (
\sigma ) ) = [ \id , \sigma ] \), where the elements of \( [ \id, \sigma
]\) are interpreted as chains when read in one-line notation.
\end{lemma}

To understand this lemma, it is helpful to look at an example. The weak
Bruhat order 
on
 \( { \mathcal S }_n\) is shown in Figure \ref{Bruhat_order_example}.
Let \( \sigma = 3 1 2\). Then, \( \Phi ( \sigma )\) has linear
extensions \( 1 \prec 2 \prec 3\), \( 1 \prec 3 \prec 2\), and \( 3
\prec 1 \prec 2\). Notice that these are exactly the elements of \( [
\id, \sigma ]\).

Before we begin considering the complexity of interval size
computations, we present a simple lemma relating permutation width and
poset width.

\begin{lemma}
\label{width}
Let \( \sigma \in { \mathcal S }_n\). Then the width of \( \sigma\) is
equal to the width of \( \Phi ( \sigma )\).
\end{lemma}
\begin{proof}
In a two-dimensional poset, a set of elements is an antichain precisely
when the elements appear in opposite orders in the two linear
extensions in its realizer. Since \( \Phi ( \sigma )\) has a realizer
consisting of \( 1 \prec 2 \prec \ldots \prec n\) and \( \sigma ( 1 )
\prec \sigma ( 2 ) \prec \ldots \prec \sigma ( n )\), the antichains in
\( \Phi ( \sigma )\) are exactly the decreasing subsequences of \(
\sigma\). Hence, the width of \( \Phi ( \sigma )\) is equal to the
width of \( \sigma\).
\end{proof}

In order to extend Wei's result on the computability of interval sizes
in the weak Bruhat order, we make use of the connection between
two-dimensional posets and intervals. We are able to handle a larger
collection of intervals because of the following theorem about the
computational complexity of linear extension counting, which was proved
in previous work. This theorem is actually a weaker version of Theorem
9 in \cite{Cooper}. We use of the full result in Section
\ref{intrinsic_width_section} of this paper, where it appears as
Theorem \ref{bounded_intrinsic_width}.

\begin{theorem}[Cooper, Kirkpatrick \cite{Cooper}]
\label{bounded_width}
Let \( P = ( S , \prec )\) be a poset with \( | S | = n\). If the width of
\( P = ( S , \prec ) \) is bounded by a fixed integer \( k\), then the
number of linear extensions of \( P = ( S , \prec )\) can be computed in \(
O ( n^{ \max ( 3 , k ) } )\) time.
\end{theorem}

As stated, Theorem \ref{bounded_width} assumes that \( k\) is fixed,
but we in fact need to understand the dependence of the constant in
the big O on \( k\). By retracing the proof of this theorem
and keeping track of the constants, we obtain the following
corollary.
\begin{corollary}
\label{bounded_width_modified}
Let \( P = ( S , \prec ) \) be a poset with \( | P | = n\) and width less
than or equal to \( k\). The number of linear extensions of \( P \) can be
computed in \( O ( k^2 n^{ \max ( 3 , k ) } )\) time, where the constant is
independent of both \( n\) and \( k\).
\end{corollary}

\begin{theorem}
\label{bounded_width_permutation_from_identity}
Let \( \sigma \in { \mathcal S }_n\) have width less than or equal to
\( k\). Let \( U = [ \id , \sigma ]\) be an interval in the weak Bruhat
order. Then \( | U |\) can be computed in \( O ( k^2 n^{ \max ( 3 , k )
} )\) time, where the constant is independent of both \( n\) and \(
k\).
\end{theorem}
\begin{proof}
By Lemma \ref{Bruhat_intervals_from_identity}, \( | U | = | { \mathcal
L } ( \Phi ( \sigma ) ) |\). From Lemma \ref{width}, the width of \(
\Phi ( \sigma )\) equals the width of \( \sigma\), which is less than
or equal to \( k\). Therefore, by Corollary
\ref{bounded_width_modified}, the number of linear extensions of \(
\Phi ( \sigma )\) can be computed in \( O ( k^2 n^{ \max ( 3 , k ) }
)\) time, and hence \( | U |\) can be computed in \( O ( k^2 n^{ \max (
3 , k ) } )\) time as well.
\end{proof}

One additional lemma allows us to generalize the previous theorem to
intervals which do not necessarily contain the identity.

\begin{lemma}
\label{arbitrary_intervals}
Let \( [ \sigma_1 , \sigma_2 ]\) be an interval in the weak Bruhat
order. Then \( | [ \sigma_1 , \sigma_2 ] | = | [ \id , \sigma_1^{ - 1 }
\sigma_2 ] |\).
\end{lemma}
\begin{proof}
This result follows immediately from Proposition 2.3 in
\cite{bjorner2}.
\end{proof}

\begin{theorem}
\label{general_bounded_width_result}
Let \( U = [ \sigma_1 , \sigma_2 ]\) be an interval in the weak Bruhat
order on \( { \mathcal S }_n\). If the width of \( \sigma_1^{ - 1 }
\sigma_2\) is less than or equal to \( k\), then \( | U |\) can be
computed in \( O ( k^2 n^{ \max ( 3 , k ) } )\) time.
\end{theorem}
\begin{proof}
By Lemma \ref{arbitrary_intervals}, \( | U | = | [ \id , \sigma_1^{ - 1
} \sigma_2 ] |\). Since the width of \( \sigma_1^{ - 1 } \sigma_2\) is
less than or equal to \( k\), by Theorem
\ref{bounded_width_permutation_from_identity}, \( | [ \id , \sigma_1^{
- 1 } \sigma_2 ] |\), and hence \( | U |\), can be computed in \( O (
k^2 n^{ \max ( 3 , k ) } )\) time.
\end{proof}

\begin{theorem}
The problem of computing the number of linear extensions of an
arbitrary two-dimensional poset and
the problem of computing the size of an arbitrary interval in the weak
Bruhat order are mutually polynomial-time reducible.
\end{theorem}
\begin{proof}
We show that a polynomial time reduction is possible in both
directions. First, suppose we are given a two-dimensional poset \( P\).
Let \( \sigma = \Phi^{-1} ( P )\). By Lemma
\ref{Bruhat_intervals_from_identity}, \( { \mathcal L } ( P ) = [ \id,
\sigma ]\). By assumption, we can compute \( | [ \id, \sigma ] |\) in
polynomial time, giving \( | { \mathcal L } ( P ) |\) in polynomial
time.

Now, suppose we are given an interval \( U = [ \sigma_1 , \sigma_2 ]\)
in the Bruhat order on \( { \mathcal S }_n\). First, we compute \(
\sigma_1^{-1} \sigma_2\). By Lemma \ref{arbitrary_intervals}, \( | U |
= | [ \id, \sigma_1^{-1} \sigma_2 ] |\). By Lemma
\ref{Bruhat_intervals_from_identity}, \( | [ \id, \sigma_1^{-1}
\sigma_2 ] | = | { \mathcal L } ( \Phi ( \sigma_1^{-1} \sigma_2 ) )
|\). A polynomial time calculation for the number of linear extensions
of \( \Phi ( \sigma_1^{-1} \sigma_2 )\) then yields the size of \( U\).
\end{proof}

\section{Permutations of Bounded Intrinsic Width }
\label{intrinsic_width_section}

We now extend the results from the previous section to
encompass a larger set of permutations, namely those of bounded
intrinsic width. Before defining intrinsic width, we need to give some
definitions related to the Gallai decomposition of a poset.

Given \(P = (S,\prec)\), define a subset \(T \subset S\) to be a {\em module} of \(P\) if, for all \(u,v \in T\) and \(x \in S \setminus T\), \(u \prec x\) iff \(v \prec x\) and \(x \prec u\) iff \(x \prec v\).  A module \(T\) is {\em strong} if, for any module \(U \subset S\), \(U \cap T \neq \emptyset\)
implies \(U \subset T\) or \(T \subset U\).  Thus, the nonempty strong modules of \(P\) form a tree order, called the {\em (Gallai) modular decomposition} of \(P\).  A strong module or poset is said to be {\em indecomposable} if its only proper submodules are singletons and the empty set.  It is a result of Gallai (\cite{Gallai}) that the maximal proper strong modules of \(P\) are a partition \(\Gal(P)\) of \(S\).
We define the quotient poset \( P / \Gal ( P )\) as the poset with
ground set consisting of the strong modules of \( P\) and partial
ordering relation defined by \( T_1 \prec T_2\) if and only if \( t_1
\prec t_2\) for some \( t_1 \in T_1\) and \( t_2 \in T_2\). For
example, consider the poset \( P = ( [9] , \prec )\) shown in Figure
\ref{Gallai_example} (a). The poset has four maximal proper strong
modules, namely \( \{ 1 , 4 , 5 , 6 \}\), \( \{ 2 , 3 \}\), \( \{ 7
\}\), and \( \{ 8 , 9 \}\). (It also has the strong, but not maximal
strong, module \( \{ 4 , 5 , 6 \}\).) Note that these strong modules
form a partition of \( P\). The quotient \( P / \Gal ( P )\) is shown
in Figure \ref{Gallai_example} (b).

\begin{figure}[htb]
\centering
\includegraphics[scale = 0.5] {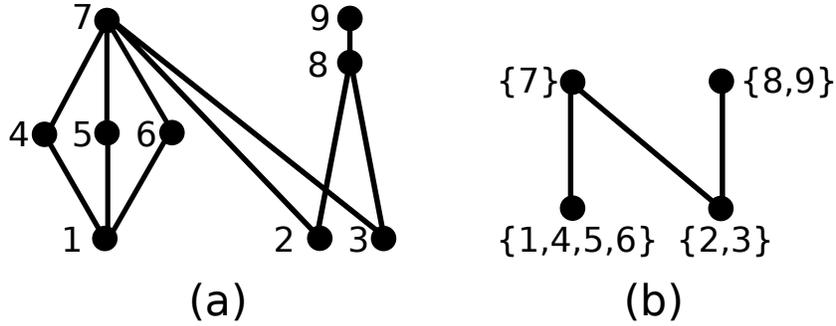}
\caption{A poset \( P \) and quotient \( P / \Gal ( P )\).}
\label{Gallai_example}
\end{figure}

 Furthermore, Gallai showed the following.  The comparability graph \(G(P)\) of a poset \(P\) has as its vertex set the ground set of \(P\) and has an edge \(\{x,y\}\) for \(x \neq y\) if and only if \(x\) and \(y\) are comparable.

\begin{theorem}[Gallai \cite{Gallai}] 
\label{Gallai_decomposition}
Given a poset \(P\) such that \(|P| \geq 2\), one of the following holds.
\begin{enumerate}
\item (Parallel-Type) If \(G(P)\) is not connected, then \(\Gal(P)\) is the family of subposets induced by the connected components of \(G(P)\) and \(P/\Gal(P)\) is an antichain.
\item (Series-Type) If the complement \(\overline{G(P)}\) of \(G(P)\) is not connected, then \(\Gal(P)\) is the family of subposets induced by the connected components of \(\overline{G(P)}\) and \(P/Gal(P)\) is a chain.
\item (Indecomposable-Type) Otherwise, \(|Gal(P)| \geq 4\) and \(P/Gal(P)\) is indecomposable.
\end{enumerate}
\end{theorem}


Figure \ref{Gallai_example} only captures the quotient construction for
the first level of the Gallai Modular Decomposition. The full
decomposition (represented only as subsets without the quotient posets)
is shown in figure \ref{full_Gallai_example}.

\begin{figure}[htb]
\centering
\includegraphics[scale = 0.36] {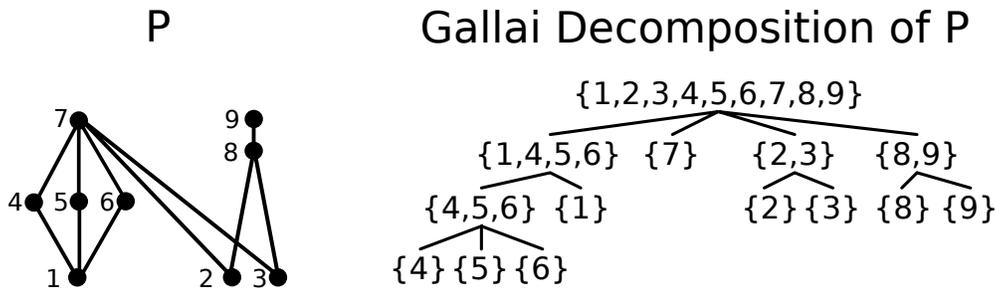}
\caption{A poset \( P \) and its Gallai decomposition.}
\label{full_Gallai_example}
\end{figure}

Define the {\em intrinsic width} \(\iw(P)\) of a poset as the maximum width of the posets \(P|_T/\Gal(P|_T)\) over all indecomposable-type nodes \(T\) of the tree order given by the Gallai modular decomposition of \(P\).  So, for example, series-parallel posets are characterized by having intrinsic width \(1\).
The example given in Figure \ref{Gallai_example} has intrinsic width \(
2\).

To parallel the concept of poset intrinsic width, we introduce the
concept of intrinsic width for permutations. First, we describe a
systematic way of building permutations, which parallels the Gallai
decomposition for posets. Let \( \sigma \in { \mathcal S }_n\) be a
permutation. Let \( ( \tau_i )_{ i = 1 }^n\) be a sequence of
permutations with \( \tau_i \in { \mathcal S }_{ m_i }\). Define \( M =
\sum_{ i = 1 }^n m_i\) and define the interval \( A_j = [ 1 + \sum_{ i
= 1 }^{ j - 1 } m_i , \sum_{ i = 1 }^{ j } m_i ]\). Finally, we define
the {\em inflation} of \( \sigma\), which is denoted \( \sigma [ \tau_1 ,
\tau_2 , \ldots , \tau_n ]\). For each \( x \in A_j\), define 
\[
\sigma [ \tau_1 , \tau_2 , \ldots , \tau_n ] ( x ) = 
\sum_{ i : \sigma(i) < \sigma(j)} m_i + 
\tau_j \left ( x - \sum_{ i = 1 }^{ j - 1 } m_i \right ).
\]

To help provide intuition for this definition, consider the following
example. Let \( \sigma = 1 32\), \( \tau_1 = 2314\), \( \tau_2 = 12\),
and \( \tau_3 = 321\). Then \( \sigma [ \tau_1 , \tau_2 , \tau_3 ] =
231489765\). This permutation is represented graphically in Figure
\ref{blowup}, and its construction as an inflation is indicated by the
three light gray boxes.

\begin{figure}[htb]
\centering
\includegraphics[scale = 0.5] {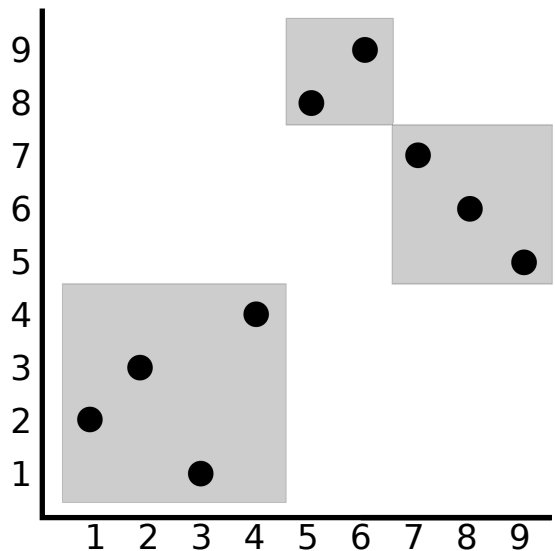}
\caption{The construction of a permutation as an inflation. The
permutation shown, \( 231489765\), is constructed as \( \sigma [ \tau_1
, \tau_2 , \tau_3 ]\) where \( \sigma = 132\), \( \tau_1 = 2314\), \(
\tau_2 = 12\), and \( \tau_3 = 321\).}
\label{blowup}
\end{figure}

We use this notion of inflation to decompose permutations in a
manner similar to the Gallai decomposition. First, we need a couple of
additional definitions. Given a permutation \( \pi \in { \mathcal S
}_n\) a {\em block} in \( \pi\) is a set of consecutive indices \( \{ i ,
i + 1 , \ldots , i + k \}\) such that the set of images \( \{ \pi ( i )
, \pi ( i + 1 ) , \ldots , \pi ( i + k ) \}\) is a contiguous subset of
\( [ n ]\). Note that blocks in the permutation \( \pi\) correspond to
intervals in the poset \( \Phi ( \pi )\). It is easy to see that all
permutations have blocks of size \( 1\) and \( n\). If \( \pi\) has no
other blocks, then it is {\em simple}. Notice that \( \pi\) is simple if
and only if \( \Phi ( \pi )\) is indecomposable. Furthermore, Albert
and Atkinson have shown that the decomposition of \( \pi = \sigma [
\tau_1 , \tau_2 , \ldots , \tau_k ]\) is unique in the case where \(
\sigma\) must be a simple permutation.

\begin{lemma}[Albert and Atkinson \cite{Albert}]
\label{block_decomposition_uniqueness}
Let \( \pi \in { \mathcal S }_n\). Then, there is a unique simple
permutation \( \sigma\) and a sequence of permutations \( \tau_1 ,
\tau_2 , \ldots , \tau_k\) such that \( \pi = \sigma [ \tau_1 , \tau_2
, \ldots , \tau_k ]\). If \( \sigma \not \in \{ 1 2 , 21 \}\), then \(
\tau_1 , \tau_2 , \ldots , \tau_k\) are also uniquely determined by \(
\pi\).
\end{lemma}

Albert and Atkinson prove this result directly by considering the
blocks of \( \pi\). However, the lemma can also be obtained by applying
the Gallai decomposition to the two-dimensional poset \( \Phi ( \pi
)\). Viewed in this light, the above lemma is in fact a restricted case
of Theorem \ref{Gallai_decomposition}. The case where \( \sigma = 12\)
corresponds to a series-type node, and the case where \( \sigma = 21\)
corresponds to a parallel-type node.  There is one minor disagreement
between the two decompositions: Albert and Atkinson would term a
monotone permutation of size more than two decomposable, whereas the
Gallai decomposition tree of its corresponding poset has height only two.  This is due
to a slight difference in the handling of series-type and parallel-type
nodes, where more than two child nodes are allowed in the Gallai
decomposition, but not in the permutation block decomposition.

Like the Gallai decomposition, we can recursively apply Lemma
\ref{block_decomposition_uniqueness} to each of \( \tau_1 , \tau_2 ,
\ldots , \tau_k\) to obtain a complete block decomposition of \( \pi\).
Since Lemma \ref{block_decomposition_uniqueness} precisely corresponds
to Theorem \ref{Gallai_decomposition}, for a given permutation \( \pi\)
the block decomposition of \( \pi\) and the Gallai decomposition of \(
\Phi ( \pi )\) have identical structure, up to the partitioning of
monotone blocks.

We define intrinsic width for permutations recursively. First, any monotone permutation has intrinsic width \(1\),
and any simple non-monotone permutation of width \( k\) has intrinsic width \( k\).
Then 
\[
\iw \left (\sigma [\tau_1 , \ldots , \tau_n ] \right ) = \max \left \{\iw(\sigma),\iw(\tau_1), \ldots , \iw(\tau_n) \right \}.
\]
Notice that this definition makes the set of all
permutations with intrinsic width bounded by \( k\) a 
``substitution-closed class''.

\begin{lemma}
\label{intrinsic_width}
Let \( \pi \in { \mathcal S }_n\). Then, \( \iw(\pi) = \iw(\Phi(\pi))\).
\end{lemma}
\begin{proof}
This follows from the correspondence between the Gallai and
block decompositions and the definition of intrinsic width. The
permutation \( \sigma\) in the definition of intrinsic width
corresponds to the quotients \( P |_T / \Gal ( P |_T )\), and the
equivalence between poset width and permutation width is provided by
Lemma \ref{width}.
\end{proof}

The following theorem, which is an extension of Theorem
\ref{bounded_width}, bounds the computational complexity
of enumerating linear extensions in the case of bounded intrinsic width. 

\begin{theorem}[Cooper, Kirkpatrick \cite{Cooper}]
\label{bounded_intrinsic_width}
If the intrinsic width of a poset is bounded by \( k\), its number of
linear extensions can be computed in \( O ( n^{ \max ( 4 , k + 1 ) }
)\) time.
\end{theorem}

By applying this to dimension-two posets, we obtain an extension of Wei's result (\cite{Wei}).

\begin{theorem}
Let \( k\) be a positive integer. Let \( U = [ \sigma_1 , \sigma_2 ]\)
be an interval in the weak Bruhat order on \( { \mathcal S }_n\). If \(
\sigma_1^{ - 1 } \sigma_2\) has intrinsic width bounded by \( k\), then
\( | U |\) can be computed in \( O ( n^{ \max ( 4 , k + 1 ) } )\) time.
\end{theorem}
\begin{proof}
By Lemma \ref{arbitrary_intervals}, 
\(
 | U | = | [ \id , \sigma_1^{ - 1} \sigma_2 ] |.
\)
By Lemma \ref{Bruhat_intervals_from_identity}, 
\[
 \left | [
\id , \sigma_1^{ - 1 } \sigma_2 ] \right | = \left | { \mathcal L } ( \Phi (
\sigma_1^{ - 1 } \sigma_2 ) ) \right |.
\]
Since \( \iw(\sigma_1^{ - 1 }
\sigma_2) \leq k\), by Lemma
\ref{intrinsic_width}, the poset \( \Phi ( \sigma_1^{ - 1 } \sigma_2
)\) also has intrinsic width bounded by \( k\). By Theorem
\ref{bounded_intrinsic_width}, \( | { \mathcal L } ( \Phi ( \sigma_1^{
- 1 } \sigma_2 ) ) |\) can be computed in \( O ( n^{ \max ( 4 , k + 1 )
} )\) time.
\end{proof}

\section{Sub-exponential Time Algorithms for Random Permutations}
\label{random_permutations_section}
Thanks to well-known results on the width of random permutations, we are
able to conclude that, for all but an exponentially small fraction of 
\( \sigma \in { \mathcal S }_n\), the quantity \( | [ \id, \sigma ] |\)
 can be computed with a sub-exponential time algorithm. We begin by 
introducing the relevant known results on random permutations. Vershik
 and Kerov \cite{Vershik}, and Logan and Shepp \cite{Logan} showed the
 following. 

\begin{theorem}
\label{mean_width_random_permutation}
Let \( L_n\) be the length of the longest monotone increasing
subsequence of a random permutation. Then,
\begin{equation}
\lim_{ n \rightarrow \infty } \frac { { \bf E } L_n } { \sqrt { n } } =
2.
\end{equation}
\end{theorem}
 
The following concentration result is due to Frieze.

\begin{theorem}[Frieze \cite{Frieze}]
\label{width_concentration}
Suppose that \( \alpha > \frac { 1 } { 3 }\). Then there exists \(
\beta = \beta ( \alpha ) > 0\) such that for \( n\) sufficiently large
\begin{equation}
{ \bf P r } ( | L_n - { \bf E } L_n | \geq n^\alpha ) \leq \exp ( -
n^\beta ).
\end{equation}
\end{theorem}

Several similar but stronger results of this type exist, but the above
theorem is sufficient for our purposes. Since permutations with an
increasing subsequence of a given length are in bijection with
permutations having a decreasing subsequence of the same length, we can
read both of these theorems as statements about expected permutation
width. Combining these results with Theorem
\ref{bounded_width_permutation_from_identity} provides the
following.

\begin{theorem}
\label{most_permutations_have_sub_exponential_algorithm}
 There exists \( \beta > 0\) so that 
there are \( n ! ( 1 - \exp ( - n^\beta ) )\) permutations \( \sigma\)
of \( n\) for which \( | [ \id, \sigma ] |\) can be computed in time \(
e^{ ( 2 + o ( 1 ) ) \sqrt { n } \log n}\).
\end{theorem}
\begin{proof}
Fix \( \frac { 1 } { 3 } < \alpha \leq \frac { 1 } { 2 }\). Then there
exists \( \beta > 0\) satisfying the conclusion of Theorem
\ref{width_concentration}. By using Theorem
\ref{mean_width_random_permutation} and selecting a sufficiently large
\( n\), we have
\begin{equation*}
\frac { | \{ \sigma \in { \mathcal S }_n : | \width ( \sigma ) - 2 \sqrt
{ n } | \geq n^\alpha \} | } { n ! } \leq \exp ( - n^\beta ).
\end{equation*}
Rearranging,
\begin{equation*}
| A | \geq n ! ( 1 - \exp ( - n^\beta ) ) ,
\end{equation*}
where
\begin{equation*}
A = \{ \sigma \in { \mathcal S }_n : | \width ( \sigma ) - 2 \sqrt { n }
| < n^\alpha \}.
\end{equation*}
 
For each \( \sigma \in A\), \( \width ( \sigma ) < 2 \sqrt { n } +
n^\alpha\). By Theorem \ref{bounded_width_permutation_from_identity},
we can compute \( | [ \id, \sigma ] |\) in time \( O ( ( 2 \sqrt { n }
+ n^\alpha )^2 n^{ \max ( 3 , 2 \sqrt { n } + n^\alpha ) } )\). Since
\( \alpha \leq \frac { 1 } { 2 }\) ,
\begin{equation*}
( 2 \sqrt { n } + n^\alpha )^2 n^{ \max ( 3 , 2 \sqrt { n } +
n^\alpha ) } = O ( n^{ 2 \sqrt { n } + n^\alpha + 1 } ) \leq e^{ ( 2 + o ( 1 ) ) \sqrt { n } \log n}.
\end{equation*}
Hence, for each \( \sigma \in A\), there is an algorithm to compute \(
| [ \id , \sigma ] |\) with the claimed time complexity.

\end{proof}

Of course, the previous theorem can be recast in terms of dimension two
posets, demonstrating that, for large enough \( n\), most
two-dimensional posets with \( n\) elements have a sub-exponential time
algorithm which computes the number of linear extensions.
\begin{corollary}
There exists \( \beta > 0\) such that, for \( n\) sufficiently large,
there are \( n ! ( 1 - \exp ( - n^\beta ) )\) two-dimensional 
naturally-labeled posets such that the number of linear
extensions can be computed in time \( e^{ ( 2 + o ( 1 ) ) \sqrt { n } \log n
}\).
\end{corollary}

\bibliographystyle{plain}
\bibliography{sources}
\end{document}